\documentclass[10pt,draft,reqno]{amsart}
      \makeatletter
     \def\section{\@startsection{section}{1}%
     \z@{.7\linespacing\@plus\linespacing}{.5\linespacing}%
     {\bfseries
     \centering
     }}
     \def\@secnumfont{\bfseries}
     \makeatother
\newtheorem{theorem}{Theorem}[section]
\newtheorem{lemma}[theorem]{Lemma}
\newtheorem{proposition}[theorem]{Proposition}
\newtheorem{corollary}[theorem]{Corollary}
\theoremstyle{definition}

\theoremstyle{remark}
\newtheorem{remark}[theorem]{Remark}
\numberwithin{equation}{section}
\def \a{{\alpha}}
\def \b{{\beta}}
\def \D{{\Delta}}

\def \d{{\delta}}
\def \e{{\varepsilon}}

\def \g{{\gamma}}

\def \o{{\omega}}

\def \p{{\varphi}}

\def \m{{\mu}}
\def \s{{\sigma}}

\def \E{{\bf E}\, }

\def \N{{\bf N}}

\def \pp{ {{\bar{\pi}}}}
\def \P{{\bf P}}

\def \qq{{\qquad}}
\def \R{{\bf R}}

\def \T{{\bf T}}

\def \Z{{\bf Z}}

\def \dd{{\rm d}}
\def \noi{{\noindent}}

\def\E{{\mathbb E \,}}

\def \T{{\mathbb T}}

\def\P{{\mathbb P}}

\def\R{{\mathbb R}}

\def\Z{{\mathbb Z}}

\def\N{{\mathbb N}}

\font\phh=cmcsc10
at  8 pt
\scrollmode

\font\gsec= cmb10 at 10 pt
   at 10,8  pt\font\phh=cmr10 at  8,2 pt
  \scrollmode
\font\gsec= cmb10 at 9,5  pt
\font\gssec= cmb10 at 8,2  pt

  \title[On small deviations of  Gaussian processes]{On Small deviations of   Gaussian processes using majorizing measures}
  \author{ Michel J.\,G. Weber}
\address{ Michel Weber: IRMA, Universit\'e
Louis-Pasteur et C.N.R.S.,   7  rue Ren\'e Descartes, 67084
Strasbourg Cedex, France. }
\email{michel.weber@math.unistra.fr
}
 \urladdr{http://www-irma.u-strasbg.fr/$\sim$weber/}
   
\begin{document}
 \maketitle
\vskip -15pt
 \centerline{---------------------------------------------------------------------------------------------------------}
\vskip - 25pt
   \begin{abstract}  We give two examples of periodic   Gaussian processes, having   entropy numbers   of exactly same order but  radically different small deviations. Our construction is based on classical Knopp's result yielding of    existence of continuous nowhere differentiable functions, and more precisely on Loud's  functions. We also obtain  a general lower bound for small deviations using the majorizing measure method. We show on examples that our bound is sharp. We also apply it to Gaussian independent sequences and to the generic class of ultrametric Gaussian processes.
    \end{abstract}
\maketitle
      \vskip 0,5cm \noi {\gssec 2010 AMS Mathematical Subject Classification}: {\phh Primary: 60F15, 60G50 ;
Secondary: 60F05}.  \par\noi  
{{\gssec Keywords and phrases}: {\phh small deviations, Gaussian processes, entropy numbers, majorizing measure method, ultrametric Gaussian processes.}} 

 \centerline{-----------------------------------------------------------------------------------------------------------}
 
  
   
 \section{Introduction--Preliminaries}  The  relatively recent  small deviations theory  of    Gaussian processes and of more
general
  processes  is a very active and interactive domain of research, having connections with statistics and operator theory.
 This is also completing the   one   of large
deviations, which has been earlier extensively investigated.
  The   large
deviations theory is essentially based on the   Borel-Sudakov-Tsirelson isoperimetric inequality, regularity methods (metric entropy
method, majorizing measure method) as well as   Slepian comparison lemma. Although some of these tools are    relevant in the   study of
  small deviations,   this one   also relies upon  
  intrinsic devices: Laplace transform, Tauberian theorems, subadditive lemma, and most importantly,    Kathri-Sid\'ak's inequality   implying for any  
Gaussian vector
$(X_1,\ldots, X_J)$ that
\begin{equation}\label{ks} \prod_{j=1}^J \P \{  |X_j|\le z \}\le \P\big\{ \sup_{j=1}^J |X_j|\le z\big\}.
\end{equation} 
  Talagrand's well-known lower bound   \cite{[T]} is based on this device. Let $\{ X(t), t\in T\}$ be a Gaussian process and let as customary $d(s,t)=\|X(s)-X(t)\|_2$, $s,t\in T$. Recall that the entropy number $N(T,d,\e)$  is the minimal number (possibly infinite) of   $d$-balls of radius $\e>0$ enough to cover $T$. Assume there exists a nonnegative function $\phi$ on $\R^+$ such that $N(T,d,\e)\le \phi (\e)$, and moreover $c_1\phi(\e)\le \phi(\frac{\e}{2}) \le c_2\phi(\e)$ for some constants $1<c_1\le c_2<\infty$. Then, for some $K>0$ and every $\e>0$,
 \begin{equation}\label{tallow} \P\big\{ \sup_{s,t\in T}|X(s)-X(t)|\le \e\big\}\ge e^{-K\phi(\e)}.
  \end{equation}
 This estimate has been recently improved in \cite{AL} where a much larger set of  size's functions $\phi$ is permitted. The basic idea is the use of inequality (\ref{ks}) to control the Laplace transform of some standard approximating chaining sum, and next to apply de Bruijn's Tauberian result. 
As moreover, some general links between the Kolmogorov's entropy  function $H(\e)=
 \log N(T,d,\e)$ (relatively to the unit ball of the associated reproducing Hilbert space) and $-\log   \P\big\{ \sup_{ t\in T}| X(t)|\le \e\big\}$ have been earlier established by   Kuelbs and Li (see \cite{LS} or \cite{Lif}),
there seems to be a kind of dual behavior between small deviations   and   entropy numbers  of a Gaussian process.
 \vskip 2pt
However, this is not exactly so. The convenient estimate  (\ref{tallow})  is  indeed  known to not always provide sharp lower estimates, whereas is some
 cases it is quite sharp. See \cite{LS} \S 3.4, \cite{Lif}  \S 2-3. A typical instance is $X(t)=g|t|^\a$, $t\in [0,1]$ where $0<\a\le1$.
We have $d (s,t) \le |t-s|^\a$, so that $N(T,d ,\e)\le C \e^{1/\a}$. However  $\P\big\{ \sup_{0\le s,t\le 1}|X(s)-X(t)|\le
\e\big\}\approx
\e$.    In fact,  much  more can be said. In section \ref{breakingd}, we show that there exist two   sample continuous periodic Gaussian processes, with  
entropy numbers   of {\it exactly} same order, but having radically different small deviations. There also exist   aperiodic sample continuous Gaussian
processes for which this duality  breakes even more dramatically.
 In Section \ref{mmlb}, we establish a new  general lower bound for small deviations by using the majorizing measure method. We show on examples that our bound is sharp. We also apply it to Gaussian independent sequences and to the generic class of ultrametric Gaussian processes.
\smallskip\par \noi {\gsec Notation and convention.} All   Gaussian processes we consider are supposed to be {\it centered}.  
The letter $g$ is used to denote  throughout   a standard Gaussian random variable. Further   $g_1, g_2, \ldots$ will
always  denote a sequence of i.i.d. standard Gaussian   random variables.   The notation $f(t)\preceq h(t)$ (resp. $f(t)\succeq h(t)$) near $ t_0\in\overline \R$ means that for $t$ in a neighborhood of $t_0$, $ |f(t)|\le c |h(t)| $, (resp. $ |f(t)|\ge c |h(t)| $) for some constant  $0<c <\infty$.
  We write $f(t)\approx h(t)$ when $f(t)\preceq h(t)$ and  $f(t)\succeq h(t)$. 

\section{Examples Breaking the Duality with Entropy Numbers} 
  \label{breakingd}
   By   considering  two kind of processes, one of   type $X(t)=gf(t)$ and the second as 
in example given in
\cite{Lif} (see (2.5),(2.6) and also \cite{LS} section 3.4), we will prove the following striking result.  
\begin{theorem} \label{twoex}Let $0<\a<1$. There exists two   cyclic continuous Gaussian processes $X_1(t), X_2(t)$, $t\ge 0$ such that, as
$\e \to 0$,
$$N([0,1],d_{X_i},\e)\approx \e^{1/\a}, \qq \quad i=1,2.$$
However,   
$$\P\{\sup_{0\le t\le 1 }|X_1(t)|\le \e \}\approx \e, \qq \log \P\{\sup_{0\le t\le 1 }|X_2(t)|\le \e \}\approx -(\log\frac1{\e})^2.$$
\end{theorem}

  Therefore  the sole knewledge of   entropy numbers of the process is,  in general   unsufficient  for   estimating its small
deviations.   
 The proof is essentially based on two lemmas.  
 
 \vskip 2pt To begin,  recall  a classical result from real analysis, namely   the existence of    continuous nowhere differentiable functions, see Knopp's construction in \cite{Kn}. In \cite{Lo}, Loud has
given an example of a function $f(t)$ which satisfies, for every real $t$, a Lipschitz condition of order precisely $\a$ ($0<\a<1$). The
proof is based on  the method used in  \cite{Kn},  as well as     van der Waerden's
construction \cite{Wa}. More precisely,  if $0<\a<1$, there exists a continuous periodic function $f$ and a pair of positive constants
$K_1,K_2$ such that
\begin{eqnarray}\label{loudpathos}    &a)&   \hbox{\it for any $t$ and any $h$, $|f(t+h)-f(t)|\le K_1|h|^\a$}, \cr
 & b)& \hbox{\it for any $t$ and infinitely many, arbitrary small  $h$},
 \cr & &\qq \qq  |f(t+h)-f(t)|\ge K_2|h|^\a .   \end{eqnarray}
  Let  $\p(t,h)$ be the saw-tooth function equal to   $0$ for even multiples of $h$, to $1$ for odd multiples of $h$ and  
linear otherwise.   Loud's   function is defined as follows: let $A$ be an integer such that $2^{2A(1-\a)}>2$ and put
 \begin{equation}\label{Loud}f(t)= \sum_{n=1}^\infty 2^{-2\a An} \p(t, 2^{-2  An}). 
\end{equation}
Then $f$ satisfies (\ref{loudpathos}). Notice that $f$ is $2^{-2  A }$--periodic. 
The
leading idea in Loud's proof  is that for every pair of values $t$ and $h$, at most one or two terms of the series (\ref{Loud}) makes a significant
contribution to the difference $f(t+h) -f(t)$. Further, it is of interest to notice that  property
$b)$ is established for  the values 
$h= 2^{-2An}$,
$n>1$.    From this and by considering $X(\o, t) =  g(\o)f(t)$,  it follows easily that 
\begin{lemma} \label{X1}
 For any $0<\a\le 1$, there exists a   cyclic Gaussian process $X(t), t\ge 0$ with sample paths verifying    a Lipschitz
condition of order precisely $\a$. Moreover, as $\e\to 0$
$$N([0,1], d_X,\e)\asymp  \e^{-\frac1{\a}}  \quad \hbox{whereas}\quad \P\{\sup_{0\le t\le 1 }|X(t)|\le \e \} \asymp \e.$$
 \end{lemma} 

Now consider the following example. Let    $0<\a<1$, $p \ge 2$ be some integer  and let $A$    integer be such that $p^{  2(1-\a)  A  }>2 $.   
      For each integer $k$,  let  $  \p_k(t)=p^{-2\a Ak} \p(t, p^{-2  Ak})  $. Put 
\begin{equation}\label{Loudprocess}f(t)= \sum_{k=1}^\infty   \p_k (t ), \qq X(t)= \sum_{k=1}^\infty g_k   \p_k (t ). 
\end{equation}
 
For proving Theorem \ref{twoex} as well as Proposition \ref{nonstat}, the  lemma below providing estimates of both the
increments of
$f$ and of its random counterpart $X$ is useful.
 \begin{lemma} \label{loudlemma} a) For all $0\le s,t\le 1$, 
$$c_1|s-t|^\a\le  \|X(s)-X(t)\|_2\le c_2|s-t|^\a, $$
where $c_1=p^{-2A  }$, $c_2= \big(\frac{ p^{   4A \a  }}{1- p^{-4(1-\a)  A }}+  \frac{1}{1-p^{-4\a A }} \big)^{1/2}$.   
\vskip 2pt
b) For all $0\le s, t\le 1$  such that $ |s-t|=   p^{-2  A (m+1) }$,    $m\ge 1$  integer,
$$ |f(s)-f(t)|\ge  \kappa_p |s-t|^\a. $$  
 And for all $0\le s,t\le 1$, 
$$ |f(s)-f(t)|\le  \mathcal K_p |s-t|^\a.$$ 
Further  $\kappa_p=p^{  -2(1-\a)  A   }\frac{1-2p^{ -2(1-\a)  A  }}{1-p^{ -2(1-\a)  A  }}$ and  $\mathcal K_p= \big(\frac{ p^{   4A \a 
}}{1- p^{-4(1-\a)  A }}+  \frac{1}{1-p^{-4\a A }} \big) $. \end{lemma} 
\begin{proof} This  is    just  reproducing Loud's proof for $p\not=2$, which we do because  the way the   constants   depend on $p$ and $\a$  matters in what follows.  Given any function $f$, we denote  for any $t$ and $\D t$,   $\Delta f= f(t+\D t)-f(t)$.     Let $m$ be the integer such that 
$p^{-2A(m+1)}<\D t\le p^{-2A m}$. The slope  of $\p_k(t)$ is $\pm\, p^{ 2(1-\a)  Ak} $, so that 
$$| \D \p_k|\le p^{ 2(1-\a)  Ak} |\D(t)|\le p^{ 2(1-\a)  Ak -2A m}
 =p^{ -2(1-\a)  A(m-k) -2A \a  m} .$$
Moreover $ \p_k$ has maximal oscillation $p^{-2\a An}$. Therefore
\begin{eqnarray} |\D f(t)|&\le&  \sum_{k=1}^\infty   | \D \p_k   (t )| \le  \sum_{k=1}^m  p^{ -2(1-\a)  A(m-k) -2A \a  m}+\sum_{k=m+1}^\infty p^{-2\a Ak}
\cr&\le&  \frac{ p^{  -2A \a  m}}{1- p^{-2(1-\a)  A }}+  \frac{p^{-2\a A(m+1)}}{1-p^{-2\a A }}
\cr&\le&  |\D t|^\a \Big(\frac{ p^{   2A \a  }}{1- p^{-2(1-\a)  A }}+  \frac{1}{1-p^{-2\a A }} \Big).
\end{eqnarray} 
And 
  \begin{eqnarray*}\|\D X(t) \|_2^2 &=&  \sum_{k=1}^\infty    [ \D \p_k   (t )]^2
  \le     \sum_{k=1}^m  p^{ -4(1-\a)  A(m-k) -4A \a  m}+\sum_{k=m+1}^\infty p^{-4\a Ak}
 \cr&\le&  \frac{ p^{  -4A \a  m}}{1- p^{-4(1-\a)  A }}+  \frac{p^{-4\a A(m+1)}}{1-p^{-4\a A }}
\cr&\le&  |\D t|^{2\a} \Big(\frac{ p^{   4A \a  }}{1- p^{-4(1-\a)  A }}+  \frac{1}{1-p^{-4\a A }} \Big).
\end{eqnarray*} 
In the other direction, fix $t$ and let $\D t= p^{-2  A (m+1) }$. By periodicity, $\D \p_k=0$ is $k>m$. And $\D \p_k= \pm\,  p^{ 2(1-\a)  Ak -2  A m }$,  if $k\le m$. Thus 
\begin{eqnarray*} |\D f(t)|    &=&    p^{   -2 A (m+1) }\big[\pm\,  p^{ 2(1-\a)  Am  }\pm\,  p^{ 2(1-\a)  A(m-1)  }\pm\ldots\pm \,  p^{ 2(1-\a)  A  }\big]
\cr &= &  p^{  -2  A  -2\a A m }\big[\pm\,  1\pm\,  p^{ -2 (1-\a)  A  }\pm\ldots\pm \,  p^{ -2(1-\a)  A (m-1) }\big].
\end{eqnarray*} 
As $ r:=p^{ -2(1-\a)  A  }<1/2$, it follows that 
$$\big|\pm\,  1\pm\,  p^{ -2 (1-\a)  A  }\pm\ldots\pm \,  p^{ -2(1-\a)  A (m-1) }\big|\ge 1- \frac{r}{1-r}=\frac{1-2r}{1-r} .$$
As $|\D t|^\a= p^{-2\a -2\a A m }$,  we therefore get
\begin{equation}  |\D f(t)| \ge  p^{  -2  A  -2\a A m }\frac{1-2p^{ -2(1-\a)  A  }}{1-p^{ -2(1-\a)  A  }} = |\D t|^\a p^{  -2(1-\a)  A   }\frac{1-2p^{ -2(1-\a)  A  }}{1-p^{ -2(1-\a)  A  }}.
\end{equation}
The corresponding estimate for $\D X$ is very easy. Let $m$ be such that $p^{-2A(m+1)}\le |\D t|<p^{-2A m }$. We have $\D \p_m   (t )= \pm p^{-2A\a m }p^{-2A m }\D t$. Thus
  \begin{eqnarray*}\|\D X(t) \|_2^2 &\ge &      [ \D \p_m   (t )]^2 =p^{-4A \a m }p^{-4A m } p^{-4A(m+1)}
   \cr &= &  p^{-4A  } p^{-4A\a m } \ge  p^{-4A  }|\D t|^{2\a}.
\end{eqnarray*} 
Hence the lower part with $c_1=p^{-2A  }$. \end{proof} 

    The following known estimate will be used.  We give a proof because it is elementary and may be easily adapted (up to some extend)  to other non
geometric coeffcients.\begin{lemma}
\label{prod}  Given any
$0<\rho<1$,
 $$\log \P\big\{ \sum_{n=1}^\infty |g_n| \rho^n\le \e\big\}\approx (\log\frac1{\e})^2, \qq as\  \e\to 0. $$
   \end{lemma}
 \begin{proof} We begin with the lower bound.  Let $H
 =\frac{\sqrt \rho}{ 1-\sqrt \rho}$. Plainly,
 $$\P\big\{ \sum_{n=1}^\infty |g_n| \rho^n\le \e_0\big\}\ge  \prod_{n=1}^\infty \P\{ |g|<  \frac{\e_0}{H}\rho^{-n/2}\}.$$
 Thus it suffices to estimate the product  $\prod_{n=1}^\infty \P\{ |g|<  \e\d^{ n }\}$ with $\e=\frac{\e_0}{H}, \d =\rho^{-1/2} $, $\d>1$. Let $a$ be such that $\P\{|g|\ge a\}\le \frac1{2}$, and put
 $N= \sup\{ n: \d^n\le \frac{a}{\e}\}$. 
 Then 
 $$\prod_{n=1}^N \P\{ |g|<  \e\d^{ n }\}\ge \P\{ |g|<  \e \}^N\ge \exp\big\{ -C_\d (\log\frac1{\e})^2\big\}.  $$
Now, 
  \begin{eqnarray*}  \sum_{n=N+1}^\infty \P\{ |g|\ge  \e\d^{ n }\}&=&\int_a^\infty \big\{ \sum_{a\le \e\d^n\le u} 1\big\} e^{-u^2/2} \dd
u\cr &\le& C_\d \int_a^\infty \big\{ 1\vee \log u\big\} e^{-u^2/2} \dd u<\infty .
\end{eqnarray*} 
As $\log (1-x)\ge -2x$ if $0\le x\le 1/2$ and $\P\{ |g|>  \e\d^{ n }\}\le 1/2$ if $n>N$,  we get 
$$ \prod_{n=N+1}^\infty \P\{ |g|<  \e\d^{ n }\}\ge  \exp\big\{ -  \sum_{n=N+1}^\infty \P\{ |g|\ge  \e\d^{ n }\}\big\}\ge c_\d>0.$$
Thus $\prod_{n=1}^\infty \P\{ |g|<  \e\d^{ n }\}\ge c_\d (\log\frac1{\e})^2$. To get the upper bound is faster. 
Let $N'=\sup\{ n: \d^n\le \frac1{\sqrt \e}\}$. Then
 \begin{eqnarray}\label{maxrho} & & \prod_{n=1}^{\infty} \P\{ |g|<  \e\d^{ n }\} \le   \prod_{n=1}^{N'} \P\{ |g|<  \e\d^{ n }\} \le
 \P\{ |g| <\sqrt
\e\}^{N'}  
  \cr &= &\exp\Big\{ -N'\log \frac1{\P\{ |g| <\sqrt \e\}} \Big\}
\le    \exp\big\{ -C_\d (\log\frac1{\e})^2\big\} .
\end{eqnarray} 
  \end{proof}\vskip 3pt
 
We can now  prove Theorem \ref{twoex}. Take $X_1 $ as in Lemma \ref{X1}. Let $p=2$ in (\ref{Loudprocess}) and
choose $X_2=X$. The entropy numbers   clearly verify
$N ([0,1], d_{X_i},
\e)
\approx
\e^{1/\a}$,
$i=1,2$. 
 First, by using Lemma \ref{prod}, $$ \P\Big\{\sup_{0\le t\le 1}|X_2(t)|\le \e \Big\}\ge  \P\Big\{\sum_{k=1}^\infty 2^{-2\a Ak}|g_k| 
\le \e \Big\}\ge e^{-C_\a
\log^2\frac1{\e}}.   $$ 
Next we notice 
\begin{eqnarray*}\p_j(2^{-2Ak})=\begin{cases}2^{-2A\a j}  \, 2^{-2A(k-j)}\quad & {\rm if} \ j\le k,\cr 
0 \quad & {\rm if} \ j> k.\end{cases}\end{eqnarray*}
Thus $X_2(2^{-2Ak})=  2^{-2A k }\sum_{j=1}^k g_j 2^{ 2  A(1-\a)j} $. And  as  
$$2^{ 2A k } X_2(2^{-2Ak})-2^{ 2A (k-1) } X_2(2^{-2A(k-1)})=g_k2^{ 2  A(1-\a)k}, $$
 it follows from (\ref{maxrho}) that  
$$  \P\Big\{\sup_{0\le t\le 1}|X_2(t)|\le \e \Big\}\le \prod_{k=1}^\infty \P\big\{ |g_k|\le  2^{ 2  A \a k}( 1+2^{-2A}) \e
\big\}\le  \exp\big\{ -C_\a (\log\frac1{\e})^2\big\} .
$$ 
This achieves the proof.
  \begin{remark}
Let  $\psi(t)=1-|2\{t\}-1|$ where $\{t\}$ denote the fractional part of $t$.   Lifshits has considered the following
example
 \begin{equation}\label{lifex}X(t) =g_0 t +\sum_{n=1}^\infty g_n 2^{- \frac{\a}{2} n} \psi ( \{ 2^n t\} )\qq t\in [0,1]. 
\end{equation}
It is observed in  \cite{Lif}   that $\|X(s)-X(t)\|_2\ge c|t-s|^{ \frac{\a}{2}}$ whereas 
$$\log \P\big\{ \sup_{  t\in T}| X(t)|\le \e\big\} \approx -\log^2\frac1{\e}.
$$
 As we said at the beginning, our second process is of   same type since $\psi(t)=\p(t,\frac1{2})  
$. 
\end{remark} 

\noi {\gsec A  class of examples.} If $\tau$ is a       piecewise   $C^2$ expanding map   on $\T= \R/\Z$,    by the Lasota-Yorke theorem
 there exists a  $\tau$-invariant probability measure  $\m$ which is absolutely continuous with respect to Lebesgue measure. So is the case for $\psi$. This leads
us to introduce the following family of processes: let  $\{a_n, n\ge 1\}\in \ell_1$, $f\in L^1(\T, \m)$ and put 
$$X(t) =\sum_{n=1}^\infty a_n g_n f(\psi^n(t)) . $$
  We  have just considered    the case $f(t)=t$. It would be certainly very informative to describe the small deviations of this
class of Gausian processes.   By Birkhoff's theorem, 
$$\frac1{n}\sum_{k=0}^{n-1} f(\psi^n(t))\to \int_\T f\dd \m \qq\quad  \hbox{almost everywhere}.$$ 
The rate of this convergence, which for specific $f$ only maybe explicited, certainly plays a role since   by using Abel summation 
  we formally have
$$X(t) =\sum_{n=1}^\infty n(a_n-a_{n+1}) g_n \big[\frac1{n}\sum_{k=0}^{n-1} f(\psi^n(t)) \big]\sim (\int_\T f\dd \m)\sum_{n=1}^\infty n(a_n-a_{n+1}) g_n  . $$

\medskip\par  
 Shao and Li     \cite{LS} argued from     example (\ref{lifex}) that  stationarity (in turn periodicity) should play a big role in  
 upper estimates for small deviations.
 
 We show that Loud's functions can be used to build   aperiodic examples breaking the
duality even more dramatically. 
 The intuitive
idea behind the construction is that adding infinitely many functions with periods
$q_n^{-1}$, where    $q_n$ are mutually coprime integers, produces   aperiodic functions.
  \begin{proposition} \label{nonstat} There exists an aperiodic sample continuous Gaussian process $\{ X(t), 0\le t\le 1\}$ such
that 
$$\lim_{\e \to 0}\frac{\log N([0,1], d_X,  
\e  ) }{\log \frac1{\e} } =\infty \quad whereas \quad 
 \liminf_{\e \to 0} \frac{\log \P\big\{\displaystyle{\sup_{0\le t\le 1}}|X(t)|\le \e \big\} }{ (\log \frac1{\e} )^2} >-\infty   .$$
 \end{proposition}
\begin{proof}  
        Let $\mathcal P$ be an infinite set of mutually coprime integers larger than $2$. Let $0<\a_p<1/2$, $\a_p\downarrow 0$.
 Take  $A=1$, then condition $p^{  2(1-\a_p)  A  }>2 $ holds. We further assume that
 \begin{equation} \label{alpha} \lim_{p\to\infty} \a_p \log p =0, \qq \quad 2^{ h p}\a_p\log p  \uparrow\infty\quad (\forall h>0).
 \end{equation}
     Let $  \p_{p,k}(t)=p^{-2\a_p  k} \p(t, p^{-2   k})  $,
 $k=1,\ldots$ and put $f_p=\sum_{k=1}^\infty  \p_{p,k}$. Then $f_p$ is $p^{-2   }$-periodic. Now let $\{ a_p, p\in \mathcal P\}$ be a
sequence of reals such that $\sum_p a_p^2<\infty$, and  consider the Gaussian process
 \begin{equation}\label{gpl} X(t)= \sum_p g_p a_p f_p(t)
 \end{equation}
  Since   $\mathcal P$ is a set of   mutually coprime integers, 
 periodicity is destroyed and so by considering its covariance, $X$ is no longer periodic. 
\vskip 2pt 
 By Lemma \ref{loudlemma},  $|f_p(s)-f_p(t)|\ge \kappa_p|s-t|^{\a_p} $, whenever  $|s-t|=p^{-2(m+1)}$, $m$ integer.  By assumption
(\ref{alpha}),  $p^{\a_p}\sim 1$ as $p\to \infty$, so that  $  
\kappa_p=p^{-2(1-\a_p)   }\frac{1- 2p^{-2(1-\a_p)}}{1-  p^{-2(1-\a_p)}} \sim p^{-2   }
$.   
 Moreover,
 $$\|f_p\|_\infty\le \sum_{k=1}^\infty p^{-2\a_p k} =\frac{p^{-2\a_p}}{1-p^{-2\a_p}}  \le \frac{1}{1-e^{-2\a_p\log p}}\le
\frac{C}{\a_p\log p}.$$
Let $0\le s,t\le 1$ be such that   $|s-t|=p^{-2(m+1)}$. Then
$$\|X(s)-X(t)\|_2^2= \sum_q a_q^2 |f_q(s)-f_q(t)|^2\ge a_p^2 |f_p(s)-f_p(t)|^2\ge  a_p^2\kappa_p^2|s-t|^{2\a_p}. $$
Thereby $\|X(s)-X(t)\|_2\ge Ca_p p ^{-2}|s-t|^{ \a_p}$.
Now let $ \a>0$. We choose $m$ integer so that 
$$m+1 \sim \frac{p\a \log 2}{2\a_p\log p}. $$
Then $|s-t|^{ \a_p}=p^{-2\a_p (m+1)}\sim 2^{-p\a}$. Let $\b,\g$ be such that  $0<\b< \a<\a+\b<\g$, and  choose $a_p=2^{-\b p}$. Then, for
all
$p$ large enough   
$$\|X(s)-X(t)\|_2\ge C2^{-(\a +\b) p} p ^{-2}\ge  2^{- \g  p}  .$$
Put $\e= 2^{- \g  p}$. Then
$$N([0,1], d_X,  
\e  )\ge   p^{2(m+1)}= e^{ 2(m+1)\log p}  \ge 2^{   c \, p  \frac{  \a  }{ \a_p }  }\gg 2^{       \frac{ p  }{ \a  }  }=  \e ^{-  
\frac{1
   }{ \a  }  }
 $$
Let    $0<\b'<\b$.  Now as $2^{ h p}\a_p\log p  \uparrow\infty  $ for any $h>0$, it follows that  
$$|X(t)|\le \sum_p |g_p|2^{-\b p}\|f_p\|_\infty\le C\sum_p |g_p|\frac{2^{-\b p} }{\a_p\log p}\le C\sum_p |g_p| 2^{-\b' p} . $$
Therefore, by using Lemma \ref{prod}
$$\P\big\{\sup_{0\le t\le 1}|X(t)|\le \e \big\}\ge \P\big\{ \sum_p |g_p| 2^{-\b' p}\le  \e /C\big\}\ge e^{-C (\log \frac1{\e})^2}.$$

\end{proof} 
\section{A General Lower Bound Using Majorizing Measures}
\label{mmlb} The results from the previous section suggest the search of lower bounds for small deviations by using the  majorizing measure method. It is known from the general theory of Gaussian processes that this is
 the paramount method for studying the  regularity  of Gaussian processes.   And also that in general,
 entropy numbers are not a sufficiently precise tool.   A classical example is provided by  independent Gaussian  sequences.  See
\cite{[T0]},\cite{[Ta]},\cite{W}.   Generally speaking,  once having Kathri-Sid\'ak's inequality in hands, the argument
leading to  
   lower bounds 
is relatively direct.  A well appreciation of the used chaining technic is however necessary. 
In \cite{W1}, we obtained a general lower estimate of   small deviations by using majorizing measure method. Since the result is 
relevant there and in the next section, we present  a slightly updated formulation of it and provide a proof. 
\vskip 2pt Let $X=\big\{ X( t),  t\in
T\big\}$ be a centered   Gaussian   process, with basic  probability space $(\Omega,\mathcal A,\P)$, and let 
  $d(s,t) =\|X(s)-X(t)\|_2$,    $D={\rm diam}(T,d) $.  We  assume that $\s=\sup_{t\in T}\|X(t)\|_2<\infty$ and that $X$ is
$d$-separable. Let  
 $ \Pi_0 \preceq \Pi_1 \preceq \ldots $ be a   sequence of finite measurable ordered  partitions of $T$ ($\Pi_{n+1}$ is a refinement of
 $\Pi_n$)  such
that 
\begin{equation}\label{assumption}\max_{\pi \in \Pi_{n }} \max_ {    u,v\in \pi }d(u,v)\le  2^{-n}D,\qq \quad n=0,1, \ldots  
\end{equation} Let
$N_n=\#\{\Pi_n\}$.   For any  $\pi\in \Pi_n$, let $\pp $  be such as
$\pi \subset \pp \in\Pi_{n-1} $.  
If $t\in T$, we also define   $ \pi_n(t) $  by  
the relations $ t\!\in\! \pi_n(t)\in \Pi_n $.
  Introduce now a majorizing measure condition. 
 \vskip 2pt \par {\it There exists a probability measure $\m$ on $T$ such that:}
\begin{equation}\label{mmc}\lim_{n\to\infty}\sup_{t\in
T}\sum_{m>n} 2^{-m}\Big(\log \frac{
1}{\m(\pi_m(t))}\Big)^{1/2} =0.
\end{equation}
The following  technical ingredient will be useful in the proof. Let $v(m)>0$ be such that $\sum_{m=0}^\infty v(m)^{-1}<\infty$, and put
$$H(n)=\sup_{t\in
T} \sum_{m>n}  (2^{-m}D)\big(\log \frac{
v(m)}{\m(\pi_m(t))}\big)^{1/2} .$$
Then $H(n)$ is finite and $H(n)\to 0$   as $n \to\infty$.   The additional term $v$ is often of
little unconvenience since, at least on standard examples,    one may take $v(m)\gg \sup_{t}\m(\pi_m(t))$, (see next section).
\begin{theorem}\label{tlower}
 For  $0<  \e\s< H(0)  $, let  $ n(\e)$ be such that
 $  H(n(\e)) \le \e\s $.
 Then, $$\P\Big\{ \sup_{ t\in T} |X(t) |\le 2  \e \s    \Big\}
\ge
 Ce^{-N_{n(\e)} (\log \frac1{\e})   }    .$$
 \end{theorem}
\begin{proof} Since $X$ is
$d$-separable, it suffices to produce a proof for a countable $d$-dense subset of $T$, which we will call again $T$.  Put  
 $$X_\pi= \int_{\pi }X(u) {\m(du)\over \m(\pi )},  \qq X_n(t)= \int_{\pi_n(t)}X(u) {\m(du)\over \m(\pi_n(t))} .$$
 These   Gaussian random variables are the bricks  of the majorizing measure method.  By (\ref{assumption}), 
$\|X_n(t)-X_{n-1}(t)\|_2\le 2^{- n}$. Elementary considerations then yield that  
  $X(t) \buildrel{a.s.}\over{=} \lim_{n\to\infty} X_n(t)$. 
  Thus $X(t)- X_n(t) \buildrel{a.s.}\over{=}  \sum_{m=n+1}^\infty 
\big(X_m(t)-X_{m-1}(t)\big)$ and we have the bound
\begin{eqnarray}\label{initbound}|X(t) |&\le& \sup_{\pi \in \Pi_n }  |X_\pi | +  \sum_{m=n+1}^\infty 
\big|X_m(t)-X_{m-1}(t)\big| 
  .
\end{eqnarray}
  Put 
$$A_m= \Big\{\forall t\in T,\frac{ |X_m(t)-X_{m-1}(t) | }{\|X_m(t)-X_{m-1}(t)\|_2}\le \big(\log \frac{
v(m)}{\m(\pi_m(t))}\big)^{1/2}\Big\}.
$$
   Then by using (\ref{initbound}) and the fact that $\|X_\pi \|_2\le \s$ for all $\pi  \in \Pi_n$ and   $n$,
\begin{eqnarray*}& &\P\Big\{  \Big\{\sup_{\pi \in \Pi_n } \frac{|X_\pi |}{\|X_\pi \|_2}\le \e\Big\}\cap
\bigcap_{m>n}A_m\Big\}
\cr &\le &\P\Big\{\forall  t\in
T,\ |X(t) |\le
\e \s +2\sum_{m>n}2^{-m}D\big(\log \frac{
v(m)}{\m(\pi_m(t))}\big)^{1/2}\Big\}
\cr &\le &\P\Big\{ \sup_{ t\in T} |X(t) |\le
\e \s + 2 \sup_{t\in
T  }\sum_{m>n}2^{-m}D\big(\log \frac{
v(m)}{\m(\pi_m(t))}\big)^{1/2}\Big\}
\cr &= &\P\Big\{ \sup_{ t\in
T} |X(t) |\le
\e \s + 2 H(n)\Big\}.
\end{eqnarray*} 

Now by   noticing that   $$A_m= \Big\{\forall \pi\in \Pi_m,\frac{ |X_{\pi}-X_{\bar\pi} | }{\|X_{\pi}-X_{\bar\pi}\|_2}\le \big(\log
\frac{ v(m)}{\m(\pi )}\big)^{1/2}\Big\} 
$$
and using Kathri-Sid\'ak's inequality (\ref{ks}),
 we get \begin{eqnarray*}& &\P\Big\{  \big\{\sup_{\pi  \in \Pi_n } \frac{|X_\pi |}{\|X_\pi \|_2}\le \e \big\}\cap
\bigcap_{m>n}A_m\Big\}
\cr &\ge &  \P \{  |g|\le \e \}^{N_n }\prod_{m>n\atop\pi \in \Pi_m}  \P\big\{  |g |\le \big(2\log \frac{
v(m)}{\m(\pi )}\big)^{1/2}\big\} 
\cr &\ge &  \exp\Big\{-N_n  \log \frac1{\P \{  |g|\le \e \}}-c\sum_{m>n\atop\pi \in \Pi_m}  \P\big\{  |g |>
\big(2\log
\frac{ v(m)}{\m(\pi )}\big)^{1/2}\Big\}
 \cr &\ge &  \exp\Big\{-N_n  \log \frac1{\P \{  |g|\le \e \}}-c\sum_{m>n }\sum_{ \pi \in \Pi_m}    \frac
{\m(\pi )}{v(m)} \Big\}
 \cr &= & \exp\Big\{-N_n  \log \frac1{\P \{  |g|\le \e \}}-c\sum_{m>n }      \frac1{v(m)}  \Big\}
 \cr &\ge  & C\exp\Big\{-N_n  \log \frac1{  \e  } \Big\}
 .
\end{eqnarray*} 
Consequently,
\begin{equation} 
 \P\Big\{ \sup_{ t\in T} |X(t) )|\le
\e \s + 2H(n)\Big\} 
    \ge   C\exp\Big\{-N_n  \log \frac1{  \e }  \Big\}.
\end{equation} 
 Choose $n=n(\e)$.
We obtained
$$\P\Big\{ \sup_{ t\in T} |X(t) |\le 2  \e \s    \Big\}
\ge
 C\exp\Big\{-N_{n(\e)} (\log \frac1{ \e} ) \Big\}    .$$
  \end{proof} 
Let $\d:[0,1]\to \R^+$ be   increasing, $\d(0)=0$, and verifying the integral condition 
$$\int_0^{D}\big(\log  \frac{
1  }{\d (u) } \big)^{1/2}\dd u<\infty .$$ 
Choose $v(m)= \frac1{\o(2^{-m }D)}$ where $\o(t)>0$ is increasing
and verifies
$\int_0^1\frac{\o(t)}{t} \dd t<\infty$.
\begin{corollary} \label{cormm}Assume  there exists  a family $\{ \Pi_{n },
n\ge 0\}$ of finite measurable ordered  partitions of $T$ satisfying (\ref{assumption})  and a probability measure $\m$ on $T$ such
that:
 $$ \min\{ \m(\pi), \pi\in \Pi_m\}\ge  \d (2^{-mD})/2 
  \qq \quad (\forall m\ge 0). $$
Let 
 $n(\e) =\sup\Big\{ n: 2\int_0^{\e_n}\big(\log  \frac{
2  }{\d^{-1}(u)\o(u)} \big)^{1/2}\dd u\le \e \s\Big\}$.  
Then
$$\P\Big\{ \sup_{ t\in T} |X(t) |\le 2  \e \s    \Big\}
\ge
 C\exp\big\{-N_{n(\e)} (\log \frac1{ \e} ) \big\}    .$$ 
\end{corollary}
\begin{proof} We have \begin{eqnarray*}  \sum_{m>n}  (2^{-m}D)\big(\log  \frac{
v(m) }{\m(\pi_m(t))} \big)^{1/2}&\le&\sum_{m>n}  (2^{-m}D)\big(\log  \frac{2
 }{\d^{-1}(2^{-m }D)\o(2^{-m }D)} \big)^{1/2}
\cr &\le & 2\int_0^{\e_n}\big(\log  \frac{
2  }{\d^{-1}(u)\o(u)} \big)^{1/2}\dd u .
\end{eqnarray*}
 Therefore
$$\P\Big\{ \sup_{ t\in T} |X(t) |\le 2  \e \s    \Big\}
\ge
 C\exp\Big\{-N_{n(\e)} (\log \frac1{ \e} ) \Big\}    .$$
\end{proof}

\noi {\gsec Example.} Consider   Gaussian
processes
$X(t), t\in [0,1]$, which satisfy the increment condition:
$$\|X(s)-X(t)\|_2\le  \d(|s-t|), \qq \quad (\forall s,t\in [0,1]).$$   
 For $m=0,1\ldots$, let $\Pi_m$ be a partition of $[0,1]$ by consecutive
intervals of length less or equal to $\e_m=\d^{-1}(2^{-m }D)$, $D=d(1)$. One can arrange it so that each interval has length greater than
$\e_m/2$. Let $\m$  be the Lebesgue measure. Then $\m(\pi) \ge  \d^{-1}(2^{-m }D)/2 $ if $\pi\in \Pi_m$. Thus  Corollary \ref{cormm}
applies. In the particular case
$d(u) =|\log u|^\b$ with
$\b>1/2$, this gives
 \begin{eqnarray} 
\log
\Big|\log\P\Big\{ \sup_{ t\in T} |X(t) |\le 2  \e \s   
\Big\}\Big|
= \mathcal O(
  \e^{-\frac{2}{2\b-1}})     .
   \end{eqnarray} 
That estimate  can also be deduced from  the very recent work \cite{AL} (Theorem 3 with $\g= \b^{-1}$), where a growth condition
on entropy numbers (namely  on the induced Gaussian metric) is given. 
\section{Gaussian Independent Sequences}
\label{appllb} Let $\p(n)\uparrow\infty$ with $n$ and consider the Gaussian sequence  $G(\p)=\{ G_n, n\in \overline \N\}$     defined by 
$$G_n= \frac{g_n}{\p(n)}, \qq G_\infty= 0. $$
  It is known (\cite{[T0]} p.102) that  already on these elementary examples, the metric entropy approach fails 
  to describe    their regularity.   As 
\begin{equation} \label{phi1}\limsup_{n\to \infty}
\frac{|g_n|}{\sqrt {2\log n}}\buildrel{a.s.}\over{=}1 ,
\end{equation}
 $G(\p)$ is sample bounded if $\p(n) = \sqrt{\log n}$,  and    sample continuous on $ \overline \N$ if and only if  \begin{equation} \label{phi}\sqrt {\log n}=o(\p(n)).
\end{equation} 
 We begin with a general remark. From Talagrand's   representation  of bounded or continuous  Gaussian processes  
 (\cite{[T0]},
theorems 2-3), we know that  a Gaussian process $\{ X(t), t\in T\}$ is sample  bounded   if and only if there exists a (not necessarily independent) Gaussian sequence  $\{\xi_n, n\ge 1\}$ with $\|\xi_n\|_2\le Ka(\log n+ {a^2}/{b^2})^{-1/2}$,  and that for each $t\in T$     one can write 
$$X(t) = \sum_{n=1}^\infty \a_n(t) \xi_n $$
  where $\a_n(t)\ge 0$, $\sum_{n=1}^\infty \a_n(t)\le 1$ and the series converges a.s. and in $L^2$. And if $T$ is a compact metric space, $\{ X(t), t\in T\}$ is sample   continuous  if and only if its covariance function is continuous,  and  the same representation  holds with   $\|\xi_n\|_2= o(\sqrt {\log n})$.   Thus by Kathri-Sid\'ak's inequality,  
$$\P\big\{\sup_{t\in T} |X(t)|\le  \e  \big\} \ge  \P\Big\{\sup_{n=1}^\infty  |\xi_n| \le  \e    \Big\}  \ge  \prod_{n=1}^\infty  \P\big\{   |\xi_n| \le  \e   \big\} .$$
 This consequently  makes the study of small deviations of sequences $G(\p)$ of particular interest in this general context. We shall show that  Theorem \ref{tlower} allows to get sharp   lower bounds. The   sequence of ordered partitions associated to $\p$ is based on an intrinsic sieve of $\N$, and as to the majorizing measure we will construct, it turns up to be very simple. 

We  notice that $\|G_n-G_m\|_2=
 ({\frac1{\p(n)^2}+\frac1{\p(m)^2}} )^{1/2}
$
and 
$$D=
\sup_{
  n,m\ge 1 }\|G_n-G_m\|_2= ({\frac1{\p(1)^2}+\frac1{\p(2)^2}} )^{1/2}, \qq \s= \sup_{
 n\ge 1 }=\|G_n\|_2=
\p_1^{-1} .$$
\begin{theorem}\label{tlower1}
 Assume that   (\ref{phi}) holds, $ \log \p(m)= \mathcal O  ( \log  m  )$ and 
\begin{equation}\int^{D}_{0}
 \big(\log
\p^{-1}(\frac1{u})\big)^{1/2}\dd u<\infty.
\end{equation} 
Let $\e_n= 2^{-n} D$    and put  $H(n) =\int^{  \e_n}_{0}
 \big(\log
\p^{-1}(\frac1{u})\big)^{1/2}\dd u  $, $n\ge 0$. For  $0<  \frac{ \e}{\p(1)}  < H(1) $, let  $ n(\e)$ be such that
 $  H(n(\e)) \le  \frac{\e}{\p(1)} $.
 There exists an absolute constant $C$ such that, 
 $$\P\Big\{ \sup_{n\ge 1} |G_n|\le    \frac{ 2}{\p(1)}  \e  \Big\}
\ge
 Ce^{-\p^{-1}(\frac1{\e_{n(\e)}}) (\log \frac1{\e})   }    .$$
  \end{theorem}
    Condition  $ \log \p(m)= \mathcal O  ( \log  m  )$ is technical. Notice that  it only excludes cases that are too regular,
typically  when $\p(m)$ increases exponentially.
\begin{proof}   Let   $F_n= \p^{-1}(\frac1{\e_n})$, $n\ge 0$. We
notice that $F_1=\p^{-1}(\p(1))=1$. For $u\ge 1$, let $\nu(u)$ denote the unique integer such that $F_{\nu(u)}\le u<F_{\nu(u)+1}$.  
\begin{lemma} \label{phi2} Let $B(u, \e )=\{ v\ge 1: \|G_u-G_v\|_2\le \e\}$. Then, 
\begin{eqnarray*}B(u, \e_n )=\{u\} \qquad\ \     & &(\forall n>\nu(u)),
\cr B(u, \e_n )\supseteq [F_{n+1}, \infty) & & (\forall
n<\nu(u)).
\end{eqnarray*}
\end{lemma}
  \begin{proof} Plainly 
 $\e_{\nu(u)+1}<\frac1{\p(u)}\le \e_{\nu(u) } $. If $n>\nu (u)$, then for any $v$, $\|G_u-G_v\|_2> \frac1{\p(u) }>\e_{\nu(u)+1}\ge
\e_n$. Hence $B(u, \e_n )=\{u\}$. Now notice that if $m\le \nu (u)$, then $v\ge F_m=\p^{-1}(\frac1{\e_m})$ implies that 
$\frac1{\p(v)}\le
\e_m$, and so 
$$\|G_u-G_v\|_2\le (\e^2_{\nu(u) }+ \e^2_m)^{1/2}\le \sqrt 2 \e_m< \e_{m-1}. $$
 Hence with $n=m-1$ the second assertion.
\end{proof}
 Let $\Pi_0=\N$.  For  $\nu\ge 1$, let $\Pi_\nu$ be the finite partition  of $\N$ defined by:   
$$\pi\in \Pi_\nu\Longleftrightarrow   \hbox{ $\pi= \{u\}$,  $u<F_\nu$   or $\pi=  [F_\nu , \infty)$.}$$ 
Then $\#\{\Pi_\nu\}= F_\nu$ and $ \Pi_{\nu +1}$ is a refinement of $\Pi_\nu$.  Further, assumption (\ref{assumption}) is satisfied since by Lemma \ref{phi2}
$$\max_{\pi \in \Pi_{\nu }} \max_ {    u,v\in \pi }d(u,v)\le  \e_{\nu} .  
$$  Let $\m$ be the probability
measure
 on $\N$ defined   by $\m\{t\}= ct^{-2}$, $c  = (\sum_{t=1}^\infty t^{-2})^{-1}$. 
Let $t\ge 1$, we set $\pi_m(t)= \{t\}$ if $t<F_m$ and $\pi_m(t)= [F_m, \infty)$ otherwise. It follows that 
\begin{equation}\m(\pi_m(t))\ge \begin{cases}Ct^{-2} & {\rm if}\quad  m>\nu(t)\cr 
 CF^{-1}_{m}  & {\rm if}\quad   m\le \nu(t). 
\end{cases}\end{equation}
Fix some integer $n$ and let $t\ge 1$. If $  n>\nu(t)$, then $t<F_n=\p^{-1}(\frac1{\e_{n }})$ and 
\begin{eqnarray*}\sum_{m=n}^\infty \e_m \big(\log \frac1{\m(\pi_m(t))}\big)^{1/2}  \le   C \big(\sum_{m=n}^\infty \e_m \big)\big(\log
t\big)^{1/2}\le C   \e_n  \big(\log
\p^{-1}(\frac1{\e_{n }})\big)^{1/2}. 
\end{eqnarray*}
Now let  $n\le \nu(t) $. If  $\nu(t)\ge m\ge n$, then  $\m(\pi_m(t))\ge CF^{-1}_{m}\ge CF^{-1}_{\nu(t)} $ and as  $t< F_{\nu(t) +1}$, we may write
\begin{eqnarray*}\sum_{m=n}^\infty \e_m \big(\log
\frac1{\m(\pi_m(t))}\big)^{1/2}& \le & C \sum_{m=n}^{\nu(t)}
\e_m
\big(\log
\p^{-1}(\frac1{\e_{m }})\big)^{1/2}
\cr & & +
\big(\sum_{m> \nu(t)} \e_m \big)  (\log t )^{1/2}
\cr &\le  & C \sum_{m=n}^{\nu(t)}
\e_m
\big(\log
\p^{-1}(\frac1{\e_{m }})\big)^{1/2} + C\e_{\nu(t)}(\log t )^{1/2}  
\cr 
   &\le  & C \sum_{m=n}^{\nu(t)+1}
\e_m
\big(\log
\p^{-1}(\frac1{\e_{m }})\big)^{1/2} 
\cr  &\le  & C \int^{ \e_n}_{\e_{\nu(t)+2}}
 \big(\log
\p^{-1}(\frac1{u})\big)^{1/2}\dd u.  
\end{eqnarray*}
  Thereby
$$ \sup_{t\ge 1}\sum_{m=n}^\infty \e_m \big(\log
\frac1{\m(\pi_m(t))}\big)^{1/2} \le C \int^{ \e_n}_{0}
 \big(\log
\p^{-1}(\frac1{u})\big)^{1/2}\dd u\to 0, $$
as $n\to \infty$, by assumption.  Condition (\ref{mmc}) is  thus realized. It remains to choose $v$.
 We first observe that if  
  $ \log v(m)= \mathcal O\big ( \log \p^{-1}(m)\big)$, then
\begin{eqnarray*}
\sum_{m>n}  2^{-m}D \big(\log \frac{
v(m)}{\m(\pi_m(t))}\big)^{1/2}&\le &
\sum_{m>n}   2^{-m}D \big(\log \frac{
1}{\m(\pi_m(t))}\big)^{1/2}\cr 
& &+ \sum_{m>n}   2^{-m}D \big(\log  v(m)  \big)^{1/2}
\cr&\le & C \int^{ \e_n}_{0}
 \big(\log
\p^{-1}(\frac1{u})\big)^{1/2}\dd u  .
\end{eqnarray*}
 
 Next, clearly   $\sum_{m=0}^\infty v(m)^{-1}<\infty$   if $v(m)=m^2$. This imposes that  $ \log m= \mathcal O\big ( \log \p^{-1}(m)\big)$ or $ \log \p(m)= \mathcal O  ( \log  m  )$, which is precisely assumed.
Consequently,  
$$H(n)=\sup_{t\in
T} \sum_{m>n}  (2^{-m}D)\big(\log \frac{
v(m)}{\m(\pi_m(t))}\big)^{1/2}=\int^{ \e_n}_{0}
 \big(\log
\p^{-1}(\frac1{u})\big)^{1/2}\dd u .$$
Let  $ n(\e)$ be such that
 $  H(n(\e)) \le   \frac{\e}{\p(1)} $. By applying Theorem \ref{tlower}, we deduce that
$$\P\Big\{ \sup_{ t\ge 1} |G_t|\le 2  \e \s    \Big\}
\ge
 Ce^{-N_{n(\e)} (\log \frac1{\e})   }    .$$ 
\end{proof} 
The following corollary easily follows.
\begin{corollary} a) Let $\p(t)= (\log t)^\b$, $\b>1/2$. Then,
$$\log \Big|\log \P\Big\{\sup_{n\ge 1} \frac{|g_n|}{\p(n)}\le \e \Big\}\Big| \preceq \e^{-\frac{2}{2\b-1}}. $$
b) Let $\p(t)= (\log t)^{\frac1{2}}(\log\log t)^{1+h}$, $ h>0$. Then,
$$\log\log  \Big|\log \P\Big\{\sup_{n\ge 1} \frac{|g_n|}{\p(n)}\le \e \Big\}\Big| \preceq \e^{-\frac1{h}}. $$

\end{corollary}
\section{Ultrametric Gaussian Processes}
For ultrametric Gaussian processes, a general upper bound of small deviations can be established.  And by using Theorem \ref{tlower}, this is completed with a sharp lower bound. A metric space $(T,d)$ is called
ultrametric when $d$ satisfies the strong triangle
inequality:
$$d(s,t)\le \max\big(d(s,u), d(u,t)\big) , \quad \qq (\forall s,t,u\in T).$$
Thus two balls of same radius are either disjoint or identical.   Let $B(t,u)= \{ s\in T: d(s,t)\le
u$,  and let $v\le u$. It also follows that  
 $s\in B(t,u)\Rightarrow B(s,v)\subset B(t,u) $. When $(T,d)$ is separable, it is
easy to show that
$(T,d)$ embeds continuously into a projective limit of sets, itself endowed with an ultrametric structure. Since we need the construction,
we briefly recall it. Let $D= {\rm diam}
(T,d)$. Let
$S_n$ be the set of centers of balls forming a minimal covering of $(T,d)$ with closed balls of radius $\e_n= 2^{-n}D$, 
 $n=0,1,\ldots$. Notice that each ball 
$B(t,\e_n)$ contains at least one element of $S_{n+1}$, thereby a ball $B(s,\e_{n+1})$ for some $s\in S_{n+1}$. Otherwise, there is one
ball
$B(t_0,\e_n)$, say, such that
$\min\{ d(t_0,s), s\in S_{n+1}\}>\e_n>\e_{n+1}$, which contradicts the fact that $S_{n+1}$ realizes a covering of $T$ of order $\e_{n+1}$.
Consider for
$n=0,1,\ldots$ the  mappings
$\theta_n :T\to S_n$, $\Pi_{n,n-1}: S_n\to S_{n-1}$ respectively defined by $ d(s,\theta_n(s))\le \e_n$ and
$ d(t,\Pi_{n,n-1}(t))\le
\e_{n-1}$. Next we define $\Pi_{n,k} :  S_n\to S_{k}$   for $n\ge k$ as follows: $\Pi_{n,n} = {\rm Id}(S_n)$ and
$$ \Pi_{n,k}=  \Pi_{n,n-1}\circ\ldots  \circ\Pi_{k+1,k}. $$
 The following  elementary   lemma arises from the construction itself, so we omit the proof. 
\begin{lemma} \label{ultrabas}The pair $\big((S_n),(\Pi_{n,k})  \big)$ defines a projective system of sets and we have
the relations
$$\theta_k=\Pi_{n,k}\circ \theta_n, \quad\qq (\forall n\ge k\ge 0).$$ 
Let $L=\lim_{{}_{{}_{\hskip -13 pt\longleftarrow}}}\big((S_n),(\Pi_{n,k})  \big)$ denote its projective limit, $G=\prod_{k=0}^\infty S_k$.
Let $\Pi_k$ be the restriction to $L$ of the projection of $G$ onto $S_k$, $k=0,1,\ldots$. Put for any two elements $s,t$ of $L$
$$\d(s,t)=  \e_{n(s,t)}, $$
where $n(s,t)= \sup \{ k\ge 0: \Pi_k(s)=\Pi_k(t)\}$. Then $(L,\d)$ is a compact ultrametric space. Moreover, the mapping $\ell:
(T,d)\to (L, \d)$ defined by $\ell(t)= \{ \theta_k(t), k\ge 0\}$ a continuous embedding from $(T, d)$ to $(L,\d)$, and we have
the relations
$$\frac1{2}\d(\ell(s),\ell(t))\le d(s,t)\le \d(\ell(s),\ell(t)), \qq (\forall s,t\in T). $$
\end{lemma}
The projective limit $L$ and thereby $T$, is easily  visualized as a tree with branches in $G$, anytwo of them  
separating at   offshoots of high "$
n(s,t)$". One can attach to any such  tree  an ultrametric Gaussian process. These classes of processes have been much
investigated by Fernique, see   \cite{F}. Let
$\{g_n, n\in
\Sigma S_k\}$ be a sequence of independent Gaussian standard random variables. We put
$$ Z(t)= \sum_{n=0}^\infty \e_ng_{\Pi_n(t)} , \qq (\forall  t\in T).$$
 \begin{theorem}  a)  For some absolute constant
$\gamma>0$, we have for   
$
\e
\le D$,
\begin{eqnarray*}\P\Big\{\sup_{s,t\in L} |Z(s)-Z(t)|  \le   \e \Big\} &\le & e^{-\gamma N(T,2 \e )}.
\end{eqnarray*}
b) Assume that condition (\ref{mmc}) is fulfilled. Then, with the notation of Theorem \ref{tlower}, letting $\s=2D/\sqrt 3   $,
$$\P\Big\{ \sup_{ t\in T} |Z(t) |\le 2  \e \s    \Big\}
\ge
 Ce^{-N_{n(\e)} (\log \frac1{\e})   }    .$$
\end{theorem}
 \begin{proof} a) The assumption made implies that from each offshoot of $S_n$  grows     at least one new branch. A plain calculation yields that
$d_Z(s,t):=\|Z(s)-Z(t)\|_2=
\e_{n(s,t)}(3/2)^{1/2}$,
$s,t\in T$.  Further, we notice that
$$Z(t)-Z(s)=
\sum_{n>n(s,t)}^\infty \e_ng_{\Pi_n(t)}. $$
 Write  $S_n= \{ s_{n,j}, 1\le j\le N_n\}$ where we set $N_n=N(T, \e_n)$. Let $L_n\subset L$, $L_n= \{ t_{n,j}, 1\le j\le
N_n\}$ be such that
$\Pi_n(t_{n,j} )=s_{n,j}$ for each $j$.   Then $\E
(Z(t_{n, i})-Z(t_{n, i-1}))^2= (3/2) \e_{n }^2  $, and since  the random variables $g_n$ are independent, we observe that
\begin{equation} \E (Z(t_{n,2i})-Z(t_{n,2i-1}))(Z(t_{n,2j})-Z(t_{n,2j-1}))= 0, \qq \quad (\forall 1\le j<i\le N_n/2).
\end{equation} So that the covariance matrix of $\{Z(t_{n,2i})-Z(t_{n,2i-1}),1\le  i\le N_n/2\}$ is diagonal with all diagonal
entries equal to
$(3/2) \e_{n }^2 $. Consequently 
\begin{eqnarray*}\P\Big\{\sup_{s,t\in L} |Z(s)-Z(t)|\le \e_{n }\Big\}&\le& \P\Big\{\sup_{1\le  i\le N_n/2} |Z(t_{n,2i})-Z(t_{n,2i-1})|\le
\e_{n }\Big\}
\cr &=& \P\Big\{\sup_{1\le  i\le N_n/2}  \frac{|Z(t_{n,2i})-Z(t_{n,2i-1})|}{\|Z(t_{n,2i})-Z(t_{n,2i-1})\|_2} \le
c\Big\} 
\cr &\le &  e^{-\gamma N(T, \e_n)},  
\end{eqnarray*}
$c$, $\gamma$ being absolute constants. Let $0<\e \le {\rm diam}(T,d)$, and let $n$ be such that $\e_{n+1}<\e \le \e_n $. Then 
\begin{eqnarray*}\P\Big\{\sup_{s,t\in L} |Z(s)-Z(t)|  \le   \e \Big\}&\le & \P\Big\{\sup_{s,t\in L} |Z(s)-Z(t)|\le \e_{n
}\Big\}\le  e^{-\gamma N(T, \e_n)}
\cr &\le & e^{-\gamma N(T,2 \e )}.
\end{eqnarray*}
b) This is    a direct consequence of Theorem \ref{tlower}.\end{proof}
 {

  \end{document}